\documentclass[11pt,reqno]{amsart}
\usepackage{amsmath}
\usepackage{amsfonts}

\numberwithin{equation}{section}
\usepackage{times}
\usepackage{amsmath,amsfonts,amstext,amssymb,amsbsy,amsopn,amsthm,eucal}
\usepackage{txfonts}
\usepackage{dsfont}
\usepackage{graphicx}   
\usepackage{hyperref}
\usepackage{accents}
\usepackage{enumerate}
\usepackage{xcolor}
\usepackage{verbatim}

\newcommand{\Rm}{{\rm Rm}}
\newcommand{\Ric}{{\rm Ric}}

\newcommand{\Vol}{{\rm Vol}}
\newcommand{\diam}{{\rm diam}}

\newcommand{\rv}{{\rm v}}

\newtheorem{theorem}{Theorem}[section]

\newtheorem{proposition}[theorem]{Proposition}
\newtheorem{lemma}[theorem]{Lemma}

\theoremstyle{definition}
\newtheorem{definition}[theorem]{Definition}
\theoremstyle{remark}
\newtheorem{remark}{Remark}[section]
\theoremstyle{remark}

\theoremstyle{remark}

\theoremstyle{remark}
\theoremstyle{remark}

\begin{document}

\title{$\epsilon$-regularity for shrinking Ricci solitons and Ricci flows}
\date{\today}
\author{Huabin Ge}
\address{Department of Mathematics, Beijingjiaotong University}
\email{hbge@bjtu.edu.cn}
\author{Wenshuai Jiang}
\address{School of mathematical Sciences,  Zhejiang University and Mathematics Institute, University of Warwick}
\email{wsjiang@zju.edu.cn}

\maketitle

\begin{abstract}
	In \cite{ChTi05}, Cheeger-Tian proved an $\epsilon$-regularity theorem for $4$-dimensional   Einstein manifolds without volume assumption. They conjectured that similar results should hold for critical metrics with constant scalar curvature, shrinking Ricci solitons, Ricci flows in $4$-dimensional manifolds and higher dimensional Einstein manifolds.  In this paper we consider all these problems. First, we construct counterexamples to the conjecture for $4$-dimensional critical metrics and counterexamples to the conjecture for higher dimensional Einstein manifolds. For $4$-dimensional shrinking Ricci solitons, we prove an $\epsilon$-regularity theorem which confirms Cheeger-Tian's conjecture with a universal constant $\epsilon$. For Ricci flow, we reduce Cheeger-Tian's $\epsilon$-regularity conjecture to a backward Pseudolocality estimate. By proving a global backward Pseudolocality theorem, we can prove a global $\epsilon$-regularity theorem which partially confirms Cheeger-Tian's conjecture for Ricci flow.  Furthermore, as a consequence of the $\epsilon$-regularity, we can show by using the structure theorem of Naber-Tian \cite{NaTi} that a collapsed limit of shrinking Ricci solitons with bounded $L^2$ curvature has a smooth Riemannian orbifold structure away from a finite number of points.	\end{abstract}



\section{Introduction}
In \cite{ChTi05}, Cheeger-Tian proved the following $\epsilon$-regularity theorem for $4$-dimensional Einstein manifolds without volume assumption.

\begin{theorem}[Cheeger-Tian]\label{t:cheeger-Tian}
	There exist universal constants $\epsilon$ and $C$ such that if $(M,g,p)$ is a $4$-dimensional Einstein manifold with $\Ric=\lambda g$, $|\lambda|\le 3$ and
	\begin{align}\label{e:L2small_Cheeger-Tian}
		\int_{B_r(p)}|\Rm|^2\le \epsilon
	\end{align}
	for some $r\le 1$, then we have curvature estimate
	\begin{align}\label{e:curvature_Cheeger-Tian}
		\sup_{B_{r/2}(p)}|\Rm|\le Cr^{-2}.
	\end{align}
\end{theorem}

The $L^2$ curvature in \eqref{e:L2small_Cheeger-Tian} is scaling invariant. The proof depends on Chopping theory developed in \cite{ChGr,CFG}, Chern-Gauss-Bonnet formula and Cheeger-Colding's almost metric cone theorem \cite{Ch,ChCo}. The constant $C$ in \eqref{e:curvature_Cheeger-Tian} has a definite lower bound which was pointed out by Hein \cite{He}. If one assume further the volume is noncollapsed, such regularity estimate is well known by  \cite{A89,BKN,Ti}.  If the integral in \eqref{e:L2small_Cheeger-Tian} is taken in the average sense:
\begin{align}\label{e:L2_small_average}
	r^4\fint_{B_r(p)}|\Rm|^2\le \epsilon,
\end{align}
such an $\epsilon$-regularity was proved by Anderson \cite{An} which is also a key ingredient in \cite{ChTi05}; see also \cite{TiVi,TiVi2,Carron} for a proof of the $\epsilon$-regularity under the condition \eqref{e:L2_small_average}.

Cheeger-Tian \cite{ChTi05} conjectured that a similar $\epsilon$-regularity as in Theorem \ref{t:cheeger-Tian} should hold on higher dimensional Einstein manifolds, $4$-dimensional shrinking solitons, Ricci flows and critical metrics which are  metrics satisfying $\Delta \Ric=\Rm\ast \Ric$; see also Definition \ref{d:critical}. In this paper we will consider these problems.


Let us recall that $(M,g,f)$ is a gradient shrinking Ricci soliton if
\begin{align}\label{e:soliton}
	\Ric_g+\nabla^2f=\lambda g \mbox{ for some constant $\lambda>0$. }
\end{align}
After normalization, we assume $f(p)\equiv \min f=0$ which is obtainable for some $p\in M$. In the rest of this paper we will use this normalization for shrinking Ricci solitons.

Shrinking Ricci Soliton is a generalization of Einstein metric. Under certain conditions on the potential function $f$, one is able to get similar results as Einstein metrics. More generally, one can consider Bakry-Emery Ricci curvature $\Ric_f\equiv \Ric+\nabla^2 f$. Under Bakry-Emery Ricci curvature lower bound, a volume comparison theorem was proved by Wei-Wylie\cite{WeWy}. Wang-Zhu \cite{WaZh} generalized Cheeger-Colding theory to  Bakry-Emery setting. Recently, Zhang-Zhu \cite{ZhZh} generalized the codimension four singularity theorem in \cite{ChNa} to the setting with bounded Bakry-Emery Ricci curvature and bounded gradient of the potential functions. The key ingredient for these generalizations is that one can use Moser's iteration to deduce Cheng-Yau gradient estimate and Li-Yau heat kernel estimate through integration by parts on the term $\nabla^2f$.

Expect the Einstein metric which is a trivial soliton, can we classify all the shrinking Ricci solitons? This is a very interesting problem which is still open. Naber \cite{Na} classified $4$-dimensional shrinking Ricci solitons under bounded curvature and nonnegative curvature operator assumptions. To approach the complete classification of shrinking Ricci solitons, several useful estimates on the potential functions were obtained by \cite{CaZh,HaMu} where they showed  universal gradient estimates for the potential functions; see also Lemma \ref{l:scalarbounds}. In \cite{MuWa}, Munteanu-Wang proved that $4$-dimensional shrinking Ricci solitons have bounded Riemann curvature provided a bounded scalar curvature condition, where the Riemann curvature bound depends on the local geometry around $p$. One key observation in \cite{MuWa} is that the Riemann curvature in dimensional four manifold is controlled by the curvature of the level set of the potential function $f$.

In this paper we will prove an $\epsilon$-regularity for shrinking Ricci solitons, the proof of which would rely on the basic Bakry-Emery geometry and an apriori potential function estimate in \cite{CaZh,HaMu}.  Our first main result of this paper is the following $\epsilon$-regularity theorem which generalizes Cheeger-Tian's $\epsilon$-regularity for Einstein metrics to shrinking Ricci solitons.

\begin{theorem}[$\epsilon$-regularity]\label{t:regularity}
	There exist universal constants $\epsilon$ and $C$ such that if $(M,g,f,p)$ is a $4$-dimensional shrinking Ricci soliton \eqref{e:soliton} with $|\lambda|\le 1$ and
	$$\int_{B_r(x)}|\Rm|^2\le \epsilon,$$
	for some $r\le 1$, then we have bounded curvature estimate
	\begin{align}\label{e:curvature_bound_reg}
		\sup_{B_{r/2}(x)}|\Rm|\le C\cdot \frac{d(p,x)^2+1}{r^2}.
	\end{align}
\end{theorem}
\vspace{.3cm}
\begin{remark}\label{r:bakry_soliton}
In \cite{Hu}, Huang proved an $\epsilon$-regularity for shrinking Ricci solitons with $\epsilon$ depending on $d(p,x)$. The proof of \cite{Hu} is similar to that of Cheeger-Tian, but he used Bakry-Emery geometry instead of Ricci curvature equation. Our proof is completely different from Huang's. Actually, as can be seen below, we reduce the proof of Theorem \ref{t:regularity} to the results of \cite{ChTi05,Li}. The key point is to control Ricci curvature. Once one can bound Ricci curvature, one can deduce Theorem \ref{t:regularity} from \cite{ChTi05,Li}. Moreover, it turns out that our approach can be used to Ricci flow where there is no Bakry-Emetry geometry.
\end{remark}

\begin{remark}\label{r:conformal_soliton}
The constant $\epsilon$ in Theorem \ref{t:regularity} is a universal constant which is independent of 	$d(p,x)$. By conformally changing to bounded Ricci curvature as in \cite{TiZh,Zh}, one can prove an $\epsilon$-regularity for shrinking Ricci soliton with $\epsilon$ depending on $d(p,x)$ as \cite{Hu}.
\end{remark}


\begin{remark}
Even one assume the metric is noncollapsed ( a lower bound for the Perelman's entropy), such curvature estimate in \eqref{e:curvature_bound_reg} is still the best estimate one can get until now,  where the curvature bound would depend on $d(p,x)$; see also \cite{HaMu}.	
\end{remark}

As a direct application of the $\epsilon$-regularity Theorem \ref{t:regularity}, we have
\begin{theorem}[Bounds on blow up points]\label{t:blowup_points}
	There are universal constants $\epsilon$ and $C$ such that if $(M,g,f,p)$ is a $4$-dimensional shrinking Ricci soliton \eqref{e:soliton} with $|\lambda|\le 1$ and $\int_M |\Rm|^2\le \Lambda$, then there exists $\{q_1,\cdots,q_N\}$ with
	\begin{align}
		N\le \Lambda \epsilon^{-1}
	\end{align}
	such that for all $x\in M$ we have
	\begin{align}
		|\Rm|(x)\le C\Big(1+d(x,p)^2\Big)\max_{1\le i\le N} \{d(x,q_i)^{-2},1\}
	\end{align}
\end{theorem}
\vspace{.25cm}
By using the structure theorem of Naber-Tian \cite{NaTi}, we can easily deduce the following convergence result for collapsed shrinking Ricci solitons.
\begin{theorem}\label{t:structure}
	Let $(M_i^4,g_i,f_i,p_i,\lambda_i)$ be a sequence of complete $4$-dimensional shrinking Ricci solitons $\Ric+\nabla^2 f_i=\lambda_i g_i$ with $|\lambda_i|\le 1$ and $\int_{M_i}|\Rm|^2\le \Lambda$. Then after passing to a subsequence there exist a length space $(X,d,p)$ and points $\{q_1,\cdots, q_N\}\subset X$ with $N\le N(\Lambda)$ such that
	\begin{enumerate}
		\item $(M_i,g_i,p_i)\stackrel{d_{GH}}{\longrightarrow} (X,d,p)$.
		\item If $x\in X\setminus \{q_\alpha,\alpha=1,\cdots,N\}$, then a neighborhood of $x$ is isometric to a smooth Riemannian orbifold.
		\item There exists a universal constant $C$ such that the sectional curvature satisfies
$$sec(x)\ge -C\Big(1+d(x,p)^2\Big)\max_{1\le i\le N} \{d(x,q_i)^{-2},1\}.$$
	\end{enumerate}	
\end{theorem}
\vspace{.3cm}
\begin{remark}
If one assume further a volume noncollapsed condition, it is well known that $X$ is a smooth orbifold. See related discussions in \cite{CaSe,ChWa,Zhx,Zh,HaMu,We}.	 Actually, by using the techniques developed in \cite{ChNa}, it was proved in \cite{ZhZh} for noncollapsed sequence that the limit space $X$ is a smooth orbifold without bounded $L^2$ curvature assumption.  It is worth pointing out that by using the techniques in \cite{ChNa,JiNa}, one is able to prove an apriori $L^2$ curvature bounds $\fint_{B_R(p_i)}|\Rm|^2\le C(\rv,R)$ where $v$ is the volume lower bound constant; see also \cite{HaMu15}.  This apriori $L^2$ curvature estimate would imply that the limit space has finitely many singular points in any compact subset of $X$.
\end{remark}

As is pointed out in Remark \ref{r:bakry_soliton} and Remark \ref{r:conformal_soliton}, one can use Bakry-Emery geometry or conformal change to get a weak $\epsilon$-regularity for shrinking Ricci solitons. However, such technique would not work for Ricci flow. Fortunately, our proof for shrinking Ricci solitons can be used to Ricci flows. The key point is an $\epsilon$-regularity estimate for Ricci flows with the average $L^2$ curvature small at a fixed time $t=0$, the proof of which depends on a backward curvature estimate.  Actually, by proving a backward Pseudolocality estimate for Riemann curvature, we can prove the following $\epsilon$-regularity for Ricci flow which partially confirms Cheeger-Tian's conjecture in the Ricci flow case.
\begin{theorem}\label{t:regularity_Ricciflow}
	There exist universal constants $\epsilon$ and $C$ such that if $(M^4,g(t))$ is a compact Ricci flow $\partial_t g=-2\Ric$ on $(-1,1)$ with bounded scalar curvature $|R|\le S$ at all time, and the curvature integral at time $t=0$ satisfies
	\begin{align}\label{e:L2_Ricciflow}
		\int_{M}|\Rm|^2\le \epsilon,
	\end{align}
	then the curvature at $t=0$ is bounded
	\begin{align}
		\sup_M |\Rm|\le C \max\{D^{-2},1\} S,
	\end{align}
	where $D=\diam(M,g(0))$.
\end{theorem}
\vspace{.4cm}

\begin{remark}
If $S=0$, which means that the flow is scalar flat and thus Ricci flat. By using Chern-Gauss-Bonnet formula one can conclude that $M$ is flat which coincides with our result. 	
\end{remark}
\begin{remark}
If one can prove a local backward pseudolocality estimate for curvature as in Proposition \ref{p:backward}, one will get a local $\epsilon$-regularity. See more discussions in Section \ref{s:discussion}.	
\end{remark}

\begin{remark}
One may relate this global estimate with Gromov's conjecture which says that a universally small $L^2$ curvature \eqref{e:L2_Ricciflow} on $(M^4,g)$ implies that $M$ exists an $F$-structure; see also \cite{Str} by using the $L^2$ curvature flow to approach such conjecture.
\end{remark}

In Section \ref{s:example} we construct examples to explain that the $\epsilon$-regularity as Theorem \ref{t:cheeger-Tian} would not hold on higher dimensional Einstein manifolds and $4$-dimensional K\"ahler metrics with zero scalar curvature which are critical metrics. In particular, we have the following
\begin{theorem}\label{t:examples}
	Cheeger-Tian's $\epsilon$-regularity conjecture doesn't hold for $4$-dimensional critical metrics or  $n$-dimensional Einstein metrics with $n\ge 5$.
\end{theorem}

\begin{remark}
By Theorem \ref{t:examples}, the $\epsilon$-regularity estimate as Theorem \ref{t:cheeger-Tian} only holds for $4$-dimensional Einstein manifolds. Therefore, one can see in \cite{ChTi05} that the Gauss-Bonnet-Chern formula plays the crucial role in the proof of the $\epsilon$-regularity Theorem \ref{t:cheeger-Tian}, since all other techniques are valid for all dimensions.
\end{remark}

\subsection*{Organizations of the paper}
In Section \ref{s:preli}, we recall some basic results about critical metrics and shrinking Ricci solitons.

In Section \ref{s:soliton}, we will prove the $\epsilon$-regularity Theorem \ref{t:regularity}, Theorem \ref{t:blowup_points} and Theorem \ref{t:structure}. First in subsection \ref{ss:eps_average_shrinking}, we will prove an $\epsilon$-regularity with average $L^2$ curvature small, the idea of which follows from \cite{Carron,GeJi,Sze}. In Subsection \ref{ss:epsilon_bounded_Ricci}, we recall the estimate in \cite{ChTi05,Li} and prove an $\epsilon$-regularity by assuming  bounded Ricci curvature. In Subsection \ref{ss:bounds_Ricci_curvature_shrinking}, by using the result in Subsection \ref{ss:epsilon_bounded_Ricci}, we show that the Ricci curvature is bounded provided the $L^2$ curvature is small. Theorem \ref{t:regularity} will follow easily from the Ricci curvature estimate in Subsection \ref{ss:bounds_Ricci_curvature_shrinking} and the $\epsilon$-regularity with a bounded Ricci curvature assumption in Subsection \ref{ss:epsilon_bounded_Ricci}.

In Section \ref{s:eps_Ricciflow}, we will consider the $\epsilon$-regularity of Ricci flow. The main estimate is the backward Pseudolocality curvature estimate in subsection \ref{ss:backward}. In Section \ref{s:example}, we will construct some examples to explain our results. In Section \ref{s:discussion}, we will discuss a few about the $\epsilon$-regularity with local $L^2$ curvature for Ricci flow.

\section{Preliminaries}\label{s:preli}
In this section we mainly recall some known results which we will use in our proof.

\subsection{Critical metrics and $\epsilon$-regularity}
In this subsection we recall the definition of critical metrics and the $\epsilon$-regularity theorem proved in \cite{TiVi,TiVi2}(see also \cite{Carron}). Let us first present the definition of critical metrics which considered by \cite{TiVi} and \cite{Carron}.
\begin{definition}[Critical metric]\label{d:critical}
	We say that a Riemannian metric is critical if its Ricci tensor satisfies a Bochner's type equality:
	\begin{align}\label{e:critical_metric}
		\Delta \Ric=\Rm\ast \Ric
	\end{align}
\end{definition}
The following metrics are critical metrics(See \cite{TiVi},\cite{Carron}):
\begin{itemize}
\item[(1)]Bach flat metric with constant scalar curvature,
\item[(2)] Self-dual or Anti-self-dual metrics with constant scalar curvature,
\item[(3)]Metric with harmonic curvature,
\item[(4)]K\"ahler metrics with constant scalar curvature.
\end{itemize}
In \cite{TiVi,Carron}, the following $\epsilon$-regularity theorem was proved for critical metrics
\begin{theorem}\label{t:eps_regularity_critical_metric}
	Let $(M^n,g,p)$ have critical metric $g$. For any $\epsilon$ there exists $\delta(n,\epsilon)$ such that if
	\begin{align}
		\sup_{B_r(x)\subset B_2(p)}r^n\fint_{B_r(x)}|\Rm|^{\frac{n}{2}}\le \delta,
	\end{align}
then we have
\begin{align}
	\sup_{B_1(p)}|\Rm|\le \epsilon.
\end{align}
\end{theorem}
\begin{remark}
The constant $\delta(n,\epsilon)$ may depend on the coefficients of equation \eqref{e:critical_metric}.	
\end{remark}
\begin{remark}
If we have Ricci curvature lower bound $\Ric\ge -(n-1)g$, by the volume comparison we only need to assume $\fint_{B_2(p)}|\Rm|^{\frac{n}{2}}\le \delta$ in Theorem \ref{t:eps_regularity_critical_metric}.
\end{remark}

\subsection{Shrinking Ricci solitons}
In this subsection, we recall some estimates of shrinking Ricci solitons which play a role in the $\epsilon$-regularity theorem. The following estimates were proved in \cite{Ch,CaZh,HaMu}.
\begin{lemma}\label{l:scalarbounds}
	Let $(M^n,g,f,p)$ be a gradient shrinking Ricci soliton \eqref{e:soliton} with normalization such that $f(p)=\min f=0$ and $|\lambda|\le 1$. Then the scalar curvature $R$ and the potential function $f$ satisfy
	\begin{align}
		0\le R(x)&\le d(x,p)^2+1\\
		|\nabla f|^2(x)&\le d(x,p)^2+1.
	\end{align}
\end{lemma}
\begin{remark}
The nonnegative scalar curvature bound was proved in \cite{Ch}. 	
\end{remark}

\subsection{Sobolev constant estimates}
In this subsection, let us recall the Sobolev constant estimate in \cite{An} which will be used several times in our proof when we apply Moser iteration.
\begin{lemma}\label{l:Sobolev}
	Let $(M^n,g,p)$ satisfy $\Ric\ge -(n-1)$ on $B_2(p)$. Then the following Sobolev inequality holds
	\begin{align}\label{e:Sobolev_An}
		\left(\fint_{B_r(p)}|u|^{\frac{2n}{n-2}}\right)^{\frac{n-2}{n}}\le C(n)r^2\fint_{B_r(p)}|\nabla u|^2,
	\end{align}
	for any $u\in C^\infty_0(B_r(p))$ with $r\le 1$.
\end{lemma}
\vspace{.3cm}
\begin{remark}
	In \cite{An}, Anderson proved an isoperimetric constant bound which gives us the Sobolev constant estimate. Recently, Dai-Wei-Zhang \cite{DWZ} generalized such isoperimetric constant estimate from lower Ricci curvature to $L^p$ Ricci curvature lower bound case for $p>n/2$.
\end{remark}

\begin{proof}
Let us simply outline a proof for Lemma \ref{l:Sobolev}.  By using Cheeger-Colding's Segment inequality \cite{ChCo,Ch,ChCo1}, we can get a Dirichlet Poincar\'e inequality \cite{Ch} and a weak Neumann type Poincar\'e inequality \cite{ChCo1}
\begin{enumerate}
	\item Dirichlet: $\fint_{B_r(p)}|u|^2\le C(n)r^2\fint_{B_r(p)}|\nabla u|^2$ for all $u\in C^\infty_0(B_r(p))$ with $r\le 1$.
	\item Neumann: $\fint_{B_r(p)}\Big| u-\fint_{B_r(p)}u \Big|^2\le C(n)r^2\fint_{B_{2r}(p)}|\nabla u|^2$ for all $u\in B_{2r}(p)$ with $r\le 1/2$.
\end{enumerate}
One can now follow the proof in \cite{HaKo} to get a Neumann type Sobolev inequality:
\begin{align}
	\left(\fint_{B_r(p)}\Big| u-\fint_{B_r(p)}u \Big|^{\frac{2n}{n-2}}\right)^{\frac{n-2}{n}}\le C(n)r^2\fint_{B_{r}(p)}|\nabla u|^2, \mbox{   for all $u\in C^\infty(B_r(p))$ with $r\le 1$.  }
\end{align}
 Combining with Dirichlet Poincar\'e inequality we get the desired Sobolev inequality \eqref{e:Sobolev_An}.
\end{proof}
\vspace{.4cm}

\section{$\epsilon$-regularity theorem of Shrinking Ricci Soliton}\label{s:soliton}
The main goal of this section is to prove the $\epsilon$-regularity theorem \ref{t:regularity} for shrinking soliton.

\subsection{$\epsilon$-regularity with small average $L^2$ curvature}\label{ss:eps_average_shrinking}
The following $\epsilon$-regularity with average $L^2$ curvature bound is a standard application of Moser iteration as \cite{An}. One can also prove this by using the idea in \cite{Carron,GeJi,Sze} without involving Moser iteration.
\begin{lemma}[$\epsilon$-regularity with average energy small]\label{l:eps_average_energy}
There exist universal constants $\epsilon$ and $C$ such that if $(M^4,g,f)$ is a gradient  Ricci soliton \eqref{e:soliton} with  $|\lambda|\le 1$, and if for some $r\le 1$ there holds $r^4\fint_{B_r(x)}|\Rm|^2\le \epsilon$ and  $r\sup_{B_r(x)}|\nabla f|\le 1$, then
\begin{align}\label{e:average_curvature}
	\sup_{y\in B_{r/2}(x)}|\Rm|^2(y)\le C\fint_{B_r(x)}|\Rm|^2.
\end{align}
\end{lemma}
\begin{remark}
This lemma holds for all dimensions by replacing $L^2$ curvature by $L^{n/2}$ curvature.	 For noncollapsing case, one can find a proof in \cite{Zh}, \cite{TiZh}, \cite{HaMu}.  For Einstein manifolds, this was proved by Anderson \cite{An} through controlling the Sobolev constant and using Moser's iteration.  We use the idea of \cite{Carron} to prove such $\epsilon$-regularity.
\end{remark}

\begin{proof}
 By scaling invariant, we may assume $r=1$. For any $y\in B_{1/2}(x)$, since $|\nabla f|$ is bounded, we can use the volume comparison of Bakry-Emery Ricci curvature \cite{WeWy} to conclude that for any $s\le 1/2$
\begin{align}
\frac{\Vol(B_s(y))}{s^4}\ge C \frac{\Vol(B_{2}(y))}{2^{4}}\ge C\Vol(B_1(x)),
\end{align}
for a universal constant $C$. Therefore, we have for any $s\le 1/2$ that
\begin{align}\label{e:subball}
s^4\fint_{B_s(y)}|\Rm|^2\le C^{-1}\frac{1}{\Vol(B_1(x))}\int_{B_r(y)}|\Rm|^2\le C^{-1}\fint_{B_1(x)}|\Rm|^2\le C^{-1}\epsilon.
\end{align}

The proof of \eqref{e:average_curvature} would have two steps. First we will show that the curvature is bounded $\sup_{B_{3/4}(x)}|\Rm|\le 16$. Based on the bounded curvature estimate, later we will improve the curvature bound from a fixed constant to the average integral bound in \eqref{e:average_curvature}.

Let us show the curvature is bounded by a constant.  Consider the function $h(y)=d(y,\partial B_1(x))^2|\Rm|(y)$ on $B_1(x)$. We will show that $\sup_{B_1(x)} h\le 1$ which will give us the curvature bound in $B_{3/4}(x)$.  Assume $h(x_0)=\sup_{B_1(x)}h=A\ge 1$ with $2r_0=d(x_0,\partial B_1(x))>0$. Then on the ball $B_{r_0}(x_0)$, we have $\sup_{B_{r_0}(x_0)}|\Rm|\le Ar_0^{-2}$ and $|\Rm|(x_0)=\frac{A}{4r_0^2}$. Let us consider the ball $B_{s}(x_0)\subset B_{r_0}(x_0)$ such that $Ar_0^{-2}= s^{-2}$. By Lemma \ref{l:derivative_curvature},
we have $\sup_{B_{s/2}(x_0)}|\nabla\Rm|\le \Lambda  s^{-3}$. Combining with the fact that $|\Rm|(x_0)\ge \frac{1}{4s^2}$, we have
\begin{align}
	\inf_{B_{s/10\Lambda}(x_0)}|\Rm|\ge \frac{1}{10s^2}.
\end{align}
Hence
\begin{align}
	\frac{s^4}{10^4\Lambda^4}\fint_{B_{s/10\Lambda}(x_0)}|\Rm|^2\ge \frac{1}{10^6\Lambda^4}.
\end{align}
By choosing $\epsilon$ small, we deduce a contradiction. Thus we have $\sup_{B_1(x)}h\le 1$. In particular, $\sup_{B_{3/4}(x)}|\Rm|\le 16.$

Denote $\mu^2\equiv \fint_{B_{3/4}(x)}|\Rm|^2\le C\epsilon$. Consider $L(y)=d(y,\partial B_{3/4}(x))^2|\Rm|(y)$ on $B_{3/4}(x)$. We will show that $\sup_{B_{3/4}(x)}L(y)\le C\mu$ which will conclude the estimate in Lemma \ref{l:eps_average_energy}.  Assume $L(y_0)\equiv K\equiv \sup_{B_{3/4}(x)} L(y)$ with $d(y_0,\partial B_{3/4}(x))=2s$. Therefore $B_{s}(y_0)\subset B_{3/4}(x)$ with $\sup_{B_s(y_0)}|\Rm|\le K s^{-2}$ and $|\Rm|(y_0)=s^{-2}K/4$ for some $K$. Let us show that $K\le C \mu$.  Indeed, since $\sup_{B_s(y_0)}|\Rm|\le \sup_{B_{3/4}(x)}|\Rm|\le 16\le s^{-2},$ we have by Lemma \ref{l:derivative_curvature} that $\sup_{B_{s/2}(y_0)}|\nabla\Rm|\le \Lambda K s^{-3}$. Thus
\begin{align}
	\inf_{B_{s/10\Lambda}(y_0)}|\Rm|\ge \frac{K}{10s^2}.
\end{align}
Therefore, by \eqref{e:subball} we get
\begin{align}
	\frac{K^2}{10^6\Lambda^4}
\le \frac{s^4}{10^4\Lambda^4}\fint_{B_{s/10\Lambda}(y_0)}|\Rm|^2\le C\mu^2
\end{align}
Hence we have $K\le C \mu$ for a universal constant $C$. Thus by \eqref{e:subball}, we get
\begin{align}
	\sup_{B_{1/2}(x)}|\Rm|^2\le C\mu^2\le C\fint_{B_1(x)}|\Rm|^2.
\end{align}
\end{proof}

\begin{lemma}\label{l:derivative_curvature}
	Let $(M^n,g,f)$ be a gradient  Ricci soliton $\Ric+\nabla^2f=\lambda g$ with $|\lambda|\le 1$. There exists a universal constant $\Lambda(n)$ such that if $\sup_{B_r(x)}|\nabla f|\le r^{-1}$ and $\sup_{B_r(x)}|\Rm|= \mu r^{-2}\le r^{-2}$ for a fixed $r\le 1$, then	
	\begin{align}
		\sup_{B_{r/2}(x)}|\nabla \Rm|\le  \frac{\Lambda\mu}{ r^{3}}.
	\end{align}
\end{lemma}
\begin{proof}
	One can prove this by using Shi's local curvature estimate. Instead of applying Shi's curvature estimate, we will use soliton equation.  Without loss of generality, let us assume $r=1$. From the soliton equation, we can compute(see \cite{MuWa})
	\begin{align}
		\Delta_{{f}} |{\Rm}|^2\ge 2|\nabla{\Rm}|^2-c(n)|{\Rm}|^3\\
		\Delta_{{f}} |\nabla{\Rm}|^2\ge 2|\nabla^2{\Rm}|^2-c(n)|\nabla {\Rm}|^2\cdot |{\Rm}|,
	\end{align}
	where $\Delta_f= \Delta -\langle \nabla \cdot , \nabla f\rangle$.
	Using Cauchy inequality, we have
		\begin{align}\label{e:deltaRm}
		\Delta |{\Rm}|^2\ge |\nabla{\Rm}|^2-c(n)|{\Rm}|^3-c(n)|\nabla {f}|^2|{\Rm}|^2\\ \label{e:deltanablaRm}
		\Delta |\nabla{\Rm}|^2\ge |\nabla^2{\Rm}|^2-c(n)|\nabla {\Rm}|^2\cdot |{\Rm}|-c(n)|\nabla{f}|^2|\nabla{\Rm}|^2.
	\end{align}
	By the assumption of bounded curvature and $|\nabla{f}|\le 1$, we have
	\begin{align}
		\Delta |\nabla{\Rm}|^2\ge -c(n)|\nabla {\Rm}|^2.
	\end{align}
	By Moser's iteration as in \cite{An}, we get
	\begin{align}\label{e:moser_nabla_curvature}
		\sup_{B_{3/4}}|\nabla {\Rm}|^2\le C(n)\fint_{B_{4/5}}|\nabla{\Rm}|^2.
	\end{align}
	By multiplying a cutoff function to \eqref{e:deltaRm} and integrating by part, we have
	\begin{align}\label{e:L2_curvature_derivative}
		\fint_{B_{4/5}}|\nabla {\Rm}|^2\le C(n)\fint_{B_{1}}|{\Rm}|^2+C(n)\fint_{B_{1}}|{\Rm}|^3\le C(n)\mu^2.
	\end{align}
	Therefore, we conclude from \eqref{e:moser_nabla_curvature} and \eqref{e:L2_curvature_derivative} that $\sup_{B_{3/4}(x)}|\nabla{\Rm}|^2\le C(n)\mu^2$ which finishes the proof.
\end{proof}

\subsection{$\epsilon$-regularity with bounded Ricci curvature}\label{ss:epsilon_bounded_Ricci}
In this subsection, we recall Cheeger-Tian's curvature estimate for manifold with bounded Ricci curvature and prove the following $\epsilon$-regularity.

\begin{theorem}\label{t:boundscurvature_assume_boundedRicc}
	There exist universal constants $\epsilon$ and $C>0$ such that if $(M,g,f)$ is a $4$-dimensional gradient shrinking Ricci soliton \eqref{e:soliton} with $|\lambda|\le 1$, and $\int_{B_1(x)}|\Rm|^2\le \epsilon$, $\sup_{B_1(x)}|\Ric|\le 3$ and $\sup_{B_1(x)}|\nabla f|\le 1$, then we have curvature bound
	\begin{align}
		\sup_{B_{1/2}(x)}|\Rm|\le C.
	\end{align}
\end{theorem}
The proof of Theorem \ref{t:boundscurvature_assume_boundedRicc} relies on the following classical result proved in \cite{Li} for bounded Ricci curvature manifolds which was originally pointed out in \cite{ChTi05}.
\begin{theorem}[\cite{ChTi05,Li}]\label{t:curvatureL2_small}
For any $\epsilon>0$ there exist constants $\delta(\epsilon)$ and $a(\epsilon)$ depending only on $\epsilon$ such that if $(M,g,x)$ is a complete $4$-dimensional manifold with $\sup_{B_r(x)}|\Ric|\le 3$ and $\int_{B_r(x)}|\Rm|^2\le \delta$ for some $r\le 1$, then  we have for some $s\ge a r$ that
\begin{align}
	s^4\fint_{B_s(x)}|\Rm|^2\le \epsilon.
\end{align}

\end{theorem}

\vspace{.3cm}
Now we are ready to prove Theorem \ref{t:boundscurvature_assume_boundedRicc}.
\begin{proof}[Proof of Theorem \ref{t:boundscurvature_assume_boundedRicc}]
This is a direct application of Theorem \ref{t:curvatureL2_small} and Lemma \ref{l:eps_average_energy}. It suffices to show for any $y\in B_{1/2}(x)$ that
\begin{align}
	|\Rm|(y)\le C.
\end{align}
Indeed, for any $\epsilon'$ there exists $\epsilon(\epsilon')$ and $a(\epsilon')$ such that by Theorem \ref{t:curvatureL2_small} if $\int_{B_{1/2}(y)}|\Rm|^2\le \int_{B_1(x)}|\Rm|^2\le \epsilon$ we have for some $s\ge a$  that
\begin{align}
	s^4\fint_{B_s(y)}|\Rm|^2\le \epsilon'.
\end{align}
Applying Lemma \ref{l:eps_average_energy}, we have for some universal constant $C$ that
\begin{align}
	|\Rm|(y)\le C\epsilon' s^{-2}\le C\epsilon' a^{-2}.
\end{align}
Since $\epsilon'$ and constant $a$ can be chosen to be universal constants, we finish the whole proof.
\end{proof}

\vspace{.5cm}
\subsection{Ricci curvature estimates}\label{ss:bounds_Ricci_curvature_shrinking}
The main purpose of this subsection is to show that the Ricci curvature is bounded under a small $L^2$ curvature integral. Our main result is the following theorem.
\begin{proposition}\label{p:boundsRicci}
	There exist universal constants $\epsilon$ and $A$ such that if $(M,g,f)$ is a $4$-dimensional gradient shrinking Ricci soliton $\Ric+\nabla^2f=\lambda g$ with $\int_{B_1(x)}|\Rm|^2\le \epsilon$ and $\sup_{B_1(x)}|\nabla f|\le 1$ and $|\lambda|\le 1$, then we have Ricci curvature bound
	\begin{align}
		\sup_{B_{1/2}(x)}|\Ric|\le A.
	\end{align}
\end{proposition}

Let us begin with the following lemma which is the key ingredient to approach the Ricci curvature estimate; see also \cite{Wa}. It derseves to point out that this is the main difference bewteen critical metrics and solitons, since we will not have similar estimate for critical metrics.
\begin{lemma}\label{l:boundsRiccibyscalar}
	Let $(M,g,f)$ be an $n$ dimensional gradient shrinking Ricci soliton $\Ric+\nabla^2f=\lambda g$ with $0<\lambda\le 1$. If $\sup_{B_1(x)}|\Rm|\le 1$ then there exists a universal constant $C(n)$ such that
	\begin{align}\label{e:Ricci_bounds_flambda}
		\sup_{B_{1/2}(x)}|\Ric|^2\le C(n)\sup_{B_1(x)}|\nabla f|+C(n)\lambda.
	\end{align}
\end{lemma}

\begin{proof}
Let us assume $\sup_{B_1(x)}|\nabla f|\le 1$, otherwise, the Ricci curvature bound \eqref{e:Ricci_bounds_flambda} holds trivially by using the Riemann curvature bound. By a direct computation using soliton equation as in \cite{MuWa}, we have
\begin{align}\label{e:deltascalar}
	\Delta_fR &=R-2|\Ric|^2\\
	\Delta_f |\Ric|^2&\ge 2|\nabla \Ric|^2+2|\Ric|^2-c(n)|\Rm|\cdot |\Ric|^2,
\end{align}
where $\Delta_f= \Delta -\langle \nabla \cdot , \nabla f\rangle$.
By the bounded assumption for $|\nabla f|$ and $|\Rm|$ and Cauchy inequality, we have on $B_1(x)$ that
\begin{align}
	\Delta |\Ric|^2\ge |\nabla\Ric|^2-c(n)|\Ric|^2.
\end{align}
Hence, by Moser's iteration which is valid due to the Sobolev constant estimate in Lemma \ref{l:Sobolev}, we get
\begin{align}\label{e:Ricci_infinitybounds}
	\sup_{B_{1/2}(x)}|\Ric|^2\le c(n)\fint_{B_{3/4}(x)}|\Ric|^2.
\end{align}
To conclude the estimate \eqref{e:Ricci_bounds_flambda}, it will suffice to bound the $L^2$ Ricci curvature in \eqref{e:Ricci_infinitybounds}. Actully, since $\sup_{B_1(x)}|\nabla f|\le 1$, by normalizing $f$ such that $f(x)=0$ which doesn't affect the estimate \eqref{e:Ricci_bounds_flambda}, we have $\sup_{B_1(x)}|f|\le 1$.

Choose a cutoff function $\varphi$ with $\varphi\equiv 1$ on $B_{3/4}(x)$ and $\varphi\equiv 0$ away from $B_1(x)$ such that $|\nabla \varphi|+|\nabla^2 \varphi|+|\nabla\Delta\varphi|\le c(n)$, which is possible since we have curvature bound and curvature derivative bound in Lemma \ref{l:derivative_curvature}.   On the other hand, by soliton equation $\Ric+\nabla^2f=\lambda g$, we have $R+\Delta f=n\lambda$. Multiplying $\varphi^2$ to \eqref{e:deltascalar} and integrating by parts with weight $e^{-f}$, we have by using $\sup_{B_1(x)}|f|\le 1$ and $\sup_{B_1(x)}|\nabla f|\le 1$ that
\begin{align}
	2\fint_{B_1(x)}|\Ric|^2\varphi^2e^{-f}&=\fint_{B_1(x)}\varphi^2R e^{-f}-\fint_{B_1(x)}\varphi^2\Delta_f R e^{-f}\\
	&\le c(n)\lambda-\fint_{B_1(x)}\varphi^2\Delta fe^{-f}+\fint_{B_1(x)}\varphi^2\Delta_f(\Delta f)e^{-f}\\
	&\le c(n)\lambda-\fint_{B_1(x)}\varphi^2|\nabla f|^2e^{-f}+\fint_{B_1(x)}\langle \nabla \varphi^2,\nabla f\rangle e^{-f}+\fint_{B_1(x)}\Delta_f\varphi^2 \Delta fe^{-f}\\
	&\le c(n)\lambda+c(n)\sup_{B_1(x)}|\nabla f|+\fint_{B_1(x)}\Delta_f\varphi^2 |\nabla f|^2e^{-f}-\fint_{B_1(x)}\langle \nabla \Delta_f\varphi^2,\nabla f\rangle e^{-f}\\
	&\le c(n)\lambda +c(n)\sup_{B_1(x)}|\nabla f|-\fint_{B_1(x)}\langle\nabla \Delta\varphi^2,\nabla f\rangle e^{-f}+\fint_{B_1(x)}\Big\langle \nabla\langle \nabla f,\nabla \varphi^2\rangle,\nabla f\Big\rangle e^{-f}.\\
	&\le c(n)\lambda +c(n)\sup_{B_1(x)}|\nabla f|+\fint_{B_1(x)}\nabla^2f(\nabla\varphi^2,\nabla f)e^{-f}
	\end{align}
	Using soliton equation $\nabla^2f=-Ric+\lambda g$, we have
	\begin{align}
	2\fint_{B_1(x)}|\Ric|^2\varphi^2e^{-f}&\le c(n)\lambda +c(n)\sup_{B_1(x)}|\nabla f|-\fint_{B_1(x)}\Ric(\nabla\varphi^2,\nabla f)e^{-f}\\
	&\le c(n)\lambda +c(n)\sup_{B_1(x)}|\nabla f|+2\fint_{B_1(x)}\varphi|\Ric|\cdot |\nabla\varphi|\cdot |\nabla f|e^{-f}\\
	&\le c(n)\lambda +c(n)\sup_{B_1(x)}|\nabla f|+\fint_{B_1(x)}\varphi^2|\Ric|^2e^{-f}+\fint_{B_1(x)} |\nabla\varphi|^2|\nabla f|^2e^{-f}\\
	&\le c(n)\lambda +c(n)\sup_{B_1(x)}|\nabla f|+\fint_{B_1(x)}\varphi^2|\Ric|^2e^{-f},
\end{align}
where we also use the fact that $|\nabla f|\le 1$.
Therefore, we get
\begin{align}
	\fint_{B_{3/4}(x)}|\Ric|^2\le c(n)\fint_{B_1(x)}|\Ric|^2\varphi^2e^{-f}\le c(n)\lambda +c(n)\sup_{B_1(x)}|\nabla f|.
\end{align}
Combining with \eqref{e:Ricci_infinitybounds} gives us
\begin{align}
	\sup_{B_{1/2}(x)}|\Ric|^2\le c(n)\lambda +c(n)\sup_{B_1(x)}|\nabla f|,
\end{align}
which finishes the proof.
\end{proof}

Now we are ready to prove Proposition \ref{p:boundsRicci}
\begin{proof}[Proof of Proposition \ref{p:boundsRicci}]
	 Consider the function $h(y)=d\Big(y,\partial B_1(x)\Big)^2|\Ric|(y)$. We only need to show $h$ is bounded in $B_1(x)$ by a universal constant. Let us argue by contradiction. Assume $h(x_0)=\sup_{B_1(x)} h= N^2>10$. Let $2r_0=d(x_0,\partial B_1(x))>0$. Then we have $\sup_{B_{r_0}(x_0)}|Ric|\le \frac{N^2}{r_0^2} $ and $|Ric|(x_0)=\frac{N^2}{4r_0^2}$. Consider the rescaled metric $\tilde{g}=\frac{N^2}{r_0^2}g$. Then the ball $(B_{r_0}(x_0),g)$ is rescaled to a ball $(B_{N}(x_0),\tilde{g})$ with $\sup_{B_N(x_0)}|\tilde{\nabla} f|\le r_0/N$, $|\tilde{\Ric}|\le 1$, $|\tilde{\Ric}|(x_0)=\frac{1}{4}$ and $\int_{B_N(x_0)}|\tilde{\Rm}|^2\le \epsilon$. In particular we have $\int_{B_1(x_0)}|\tilde{\Rm}|^2\le \epsilon$. Moreover we have shrinking soliton
	 \begin{align}
	 	\tilde{\Ric}+\tilde{\nabla}^2 f=\tilde{\lambda}\tilde{g}
	 \end{align}
	with $\tilde{\lambda}=\lambda \frac{r_0^2}{N^2}\le \frac{r_0^2}{N^2}$. Applying Theorem \ref{t:boundscurvature_assume_boundedRicc} to $(B_1(x_0),\tilde{g})\subset B_N(x_0)$, we conclude that $\sup_{B_{1/2}(x_0)}|\tilde{\Rm}|\le C$ for some universal constant $C$. On the other hand, by Lemma \ref{l:boundsRiccibyscalar}, we can use $|\tilde{\nabla}f|+|\tilde{\lambda}|$ to control the Ricci curvature such that
	\begin{align}\label{e:tilde_Ric_bounds}
	\sup_{B_{1/4}(x_0)}|\tilde{\Ric}|^2\le \tilde{C}\Big(\sup_{B_{1}(x_0)}|\tilde{\nabla}f|+|\tilde{\lambda}|\Big)\le \tilde{C}\frac{r_0}{N},
	\end{align}
	for a universal constant $\tilde{C}$. Noting that $r_0\le 1$, if $N$ is large enough, this contradicts with $|\tilde{\Ric}|(x_0)=\frac{1}{4}$. In fact, $|\tilde{\Ric}|(x_0)=1/4$ and \eqref{e:tilde_Ric_bounds} give the bound $N\le \tilde{C}r_0\le \tilde{C}$. Therefore, we get our desired Ricci curvature bound.
\end{proof}

\vspace{.5cm}
\subsection{Proving Theorem \ref{t:regularity}}
Theorem \ref{t:regularity} will follow from Proposition \ref{p:boundsRicci} and Theorem \ref{t:boundscurvature_assume_boundedRicc}. Indeed, for any $x\in M$ and $r\le 1$ with $d(x,p)=R$ we have by Lemma \ref{l:scalarbounds} that $|\nabla f|^2\le R^2+1$ on $B_{r}(x)$.  Without loss any generality, assume $R>1$. To prove Theorem \ref{t:regularity},  it suffices to show that there exists a universal constant $C$ such that for every point $y\in B_{r/2}(x)$
\begin{align}
	|\Rm|(y)\le Cr^{-2}R^2.
\end{align}

Actually, consider the rescaled metric $\tilde{g}=4r^{-2}R^2g$. We have $|\tilde{\nabla} f|^2\le 1$ on $\tilde{B}_{2R}(y)$ and $\int_{B_{r/2}(y)}|\Rm|^2=\int_{\tilde{B}_R(y)}|\tilde{\Rm}|^2$. By Proposition \ref{p:boundsRicci}, there exist universal constants $\epsilon$ and $A$ such that if $\int_{\tilde{B}_1(y)}|\tilde{\Rm}|^2\le \int_{\tilde{B}_R(y)}|\tilde{\Rm}|^2\le \epsilon$ we have $\sup_{\tilde{B}_{1/2}(y)}|\tilde{\Ric}|\le A$. Rescaling $\tilde{g}$ by $\hat{g}=A\tilde{g}=Ar^{-2}R^2g$ and applying Theorem \ref{t:boundscurvature_assume_boundedRicc} we get for some universal constant $C$ that
\begin{align}
	\sup_{\hat{B}_{1/2}(y)}|\hat{\Rm}|\le C.
\end{align}
 Returning back to metric $g$ we have
\begin{align}
	|\Rm|(y)\le CAr^{-2}R^2,
\end{align}
where $C$ and $A$ are universal constants. Hence we finish the proof of Theorem \ref{t:regularity}.\qed

\subsection{Proving Theorem \ref{t:blowup_points}}
The proof is standard. For any $i\ge 0$ define the set $A_i=\{x_1^i,\cdots, x_{k_i}^i\}\subset M$ to be a maximal subset of points such that $\int_{B_{2^{-i}}{(x)}}|\Rm|^2\ge \epsilon$ and $ B_{2^{-i}}{(x)}$ are disjoint, where $\epsilon$ is the universal constant in Theorem \ref{t:regularity}. It is obvious that $A_i\subset B_2(A_0)$ which is uniformly bounded.

Denote $r_i=2^{-i}$. For each $i$ let us refine the set $A_i$ by the following process. Add minimal number of points of $A_{i-1}$ to $A_i$ such that the new set ${A}_{i,i-1}$ satisfying $B_{2r_{i-1}}(A_{i-1})\subset B_{4r_{i-1}}({A}_{i,i-1})$. It is easy to check that the new point $x\in A_{i,i-1}\cap A_{i-1}$ satisfying that $B_{r_{i-1}}(x)\cap B_{r_i}(A_i)=\emptyset$ which implies the cardinality of $A_{i,i-1}$ is less than $\Lambda\epsilon^{-1}$.

Then add minimal number of points of $A_{i-2}$ to ${A}_{i,i-1}$ such that the new set ${A}_{i,i-2}$ satisfying $B_{2r_{i-2}}(A_{i-2})\subset B_{4r_{i-2}}({A}_{i,i-2})$.  By doing this process $i$ times, we get a new set ${A}_{i,0}$ which satisfies $B_{2r_j}(A_j)\subset B_{4r_j}({A}_{i,0})$ for all $0\le j\le i$.  From the construction we have the cardinality of $\# {A}_{i,0}$ is less than $\Lambda\epsilon^{-1}$. Since $A_i\subset B_2(A_0)$ we have for some sufficiently large $i_0$ that $A_{i_0}=\emptyset$. Set  ${A}=A_{i_0,0}=\{p_1,\cdots,p_N\}$. Then the finite set $A$ satisfies Theorem \ref{t:blowup_points}.

Indeed, for any $x\in M$ if $r=d(x,A)$ with $4r_{i+1}<r\le 4r_i$, then $x\notin B_{2r_{i+1}}(A_{i+1})$ which implies that $\int_{B_{r_{i+1}}(x)}|\Rm|^2\le \epsilon$. Hence by Theorem \ref{t:regularity} we have
\begin{align}
	|\Rm|(x)\le C\Big(d^2(p,x)+1\Big) (r^{-2}+1),
\end{align}
which concludes the estimate in Theorem \ref{t:blowup_points}.\qed



\vspace{.4cm}
\subsection{Proving Theorem \ref{t:structure}}
By the volume comparison in \cite{WeWy} and the potential estimate in Lemma \ref{l:scalarbounds}, after passing to a subsequence we have $(M_i,g_i,p_i)\to (X,d,p)$ in pointed Gromov-Hausdorff sense for some length space $(X,d,p)$. By Theorem \ref{t:blowup_points} we have finite points $A_i=\{q_1^i,\cdots, q_{N_i}^i\}\subset M_i$ with $N_i\le N(\Lambda)$ and
\begin{align}
	|\Rm|(x)\le C\Big(d^2(p_i,x)+1\Big) (d(x,A_i)^{-2}+1), \mbox{  for all $x\in M_i$}.
\end{align}
Consider the limit $A_\infty\subset X$ of $A_i$.  The cardinality $\# A_\infty \le N(\Lambda)$. For any $x\in X\setminus A_\infty$ with $r=d(x,A_\infty)>0$, assume $x_i\in M_i\to x$ which implies that $d(x_i,A_i)>r/2$ for sufficiently large $i$. By Theorem \ref{t:regularity}, we have
\begin{align}
	\sup_{B_{r/2}(x_i)}|\Rm|\le C(d(p,x)) r^{-2}.
\end{align}
Combining with the soliton equation, we have for all $k\ge 0$ as in Lemma \ref{l:derivative_curvature}  that it holds the following curvature derivative estimate
\begin{align}
	\sup_{B_{r/4}(x_i)}|\nabla^k \Rm|\le C_k r^{-2-k}.
\end{align}
 By Theorem 1.1 and Remark 1.2 of Naber-Tian \cite{NaTi}, we conclude that a neighborhood of $x$ is a smooth Riemannian orbifold. The lower sectional curvature bounds is a direct consequence of the fact the Alexandroff space is preserved along GH-convergence. See also the discussion in Naber-Tian \cite{NaTi} and \cite{LoVi}.\qed

\vspace{.5cm}
\section{$\epsilon$-regularity of Ricci flow}\label{s:eps_Ricciflow}
The main purpose of this section is to prove the $\epsilon$-regularity of Ricci flow. The idea is similar as that of the shrinking soliton case where we reduce the $\epsilon$-regularity to bounded Ricci curvature manifold case. We will see that the crucial part is a backward Pesudolocality estimates for Ricci flow. The main result is the following theorem
\begin{theorem}\label{t:eps_regular_Ricciflow_global}
	There exist universal constants $\epsilon$ and $C$ such that if $(M^4,g(t))$ is a compact Ricci flow on $(-1,1)$ with bounded scalar curvature $|R|\le 1$ at all time, and if at the time $t=0$ we have
	\begin{align}
		\int_{M}|\Rm|^2\le \epsilon
	\end{align}
	and the diameter $\diam(M,g(0))=D$, then the curvature at $t=0$ is bounded
	\begin{align}
		\sup_M |\Rm|\le C \max\{D^{-2},1\}.
	\end{align}
\end{theorem}

\begin{remark}
	By scaling $|R|\le S$ to $|R|\le 1$ we can see that Theorem \ref{t:regularity_Ricciflow} is a direct corollary of Theorem \ref{t:eps_regular_Ricciflow_global}. To prove theorem \ref{t:eps_regular_Ricciflow_global} we may assume $D\ge 1$ since we can rescal the metric when $D<1$.
\end{remark}

\subsection{Curvature Derivative estimates}\label{ss:backward}
In this subsection, we will show that Ricci flow equips properties similar as that of critical metrics.  The main result of this subsection is the following curvature derivative estimate which is one key estimate toward our $\epsilon$-regularity theorem. \begin{theorem}\label{t:curvature_derivative}
	Let $(M^n,g(t))$ be a compact Ricci flow on $(-1,1)$ with bounded scalar curvature $|R|\le n$. If the Riemann curvature $\sup_{(M,g(0))}|\Rm|\le A$ with $A\ge 1$ and $\diam(M,g(0))\ge 1$, then
	\begin{align}
		\sup_{(M,g(0))}|\nabla \Rm|\le C(n)A^{3/2}.
	\end{align}
\end{theorem}
\vspace{.3cm}
\begin{remark}
The proof of Theorem \ref{t:curvature_derivative} depends on a backward curvature estimate and Shi's curvature derivative estimates. All the results in this subsection hold for all dimensions.
\end{remark}

\subsubsection{Backward Curvature estimates}

Applying similar argument as in Bamler-Zhang \cite{BaZh}, we can prove the following backward curvature estimate. See also Li-Wang-Zheng \cite{LWZ} for a backward curvature estimate of Calabi flow.

\begin{proposition}[Backward Curvature]\label{p:backward}
	Let $(M^n,g(t))$ be a compact Ricci flow on $(-1,1)$ with bounded scalar curvature $|R|\le n$. If the Riemann curvature $\sup_{(M,g(0))}|\Rm|\le A$ with $A\ge 1$ and $\diam(M,g(0))\ge 1$, then there exists a constant $\epsilon(n)>0$ such that $\sup_{(M,g(t))}|\Rm|\le 2A$ for all $-\epsilon A^{-1}\le t\le 0$.
\end{proposition}

The proof depends on an improved curvature estimate on Ricci flow. The following lemma could be found in Bamler-Zhang \cite{BaZh}; see also Wang \cite{Wa}.  We only sketch a proof here, see \cite{BaZh} for more details.
\begin{lemma}[Improved curvature estimates]\label{l:time_derivative_estimate}
	Let $(M^n,g_t)$ be a Ricci flow on $(-1,1)$ with bounded scalar curvature  $|R|\le S\le n$ and $\diam(M,g(0))\ge 1$. Then there exists universal $C(n)$ such that if
		\begin{align}\label{e:bounded_curvature_assumption}
		\sup_{-r^2\le t\le 0}\sup_{B_r(x,g_t)}|\Rm|\le r^{-2}, \mbox{ for some $r\le 1$,}
	\end{align}
	then
	\begin{align}
	\sup_{-r^2/2\le t\le 0}\sup_{B_{ r/2}(x,g_t)}|\Ric|&\le C \sqrt{S}r^{-1}\\
		\sup_{-r^2/2\le t\le 0}\sup_{B_{ r/2}(x,g_t)}\Big|\partial_t|\Rm|\Big|&\le C\sqrt{S}r^{-3}.
	\end{align}
\end{lemma}
\begin{remark}
$B_r(x,g_t)$ is a ball with radius $r$ in $(M,g(t))$. By distance control due to bounded curvature assumption \eqref{e:bounded_curvature_assumption}, one can replace the region $P_{r,x,t}\equiv \cup_{-r^2\le t\le 0}B_r(x,g_t)$ by the space time region  $[-r^2,0]\times B_r(x,g_0)$.
\end{remark}

\begin{proof}
	
Let us first recall the evolution equation for scalar curvature
	$$\partial_t R=\Delta R+2|\Ric|^2.$$
	By multiplying a cutoff function to the above equation and integrating on space time $(-r^2,0]\times B_r(x,g_t)$, we get
	\begin{align}
		\fint_{-3r^2/4}^{0}dt \fint_{B_{3r/4}(x,g_t)}|\Ric|^2\le C(n)r^{-2}\sup_{(-r^2,0]\times B_r(x,g_t)}|R| \le C(n){S}r^{-2},
	\end{align}
	where we need to use bounded curvature estimate to get the derivative estimates for cutoff function.
	Now by using the evolution equation of $\Ric$
	\begin{align}\label{e:evolution_Ricci}
		(\partial_t-\Delta-2\Rm)\Ric=0,
	\end{align}
	and by parabolic Moser iteration which is valid due to Sobolev constant estimate in Lemma \ref{l:Sobolev}, we deduce the  pointwise Ricci curvature estimate
	\begin{align}
		\sup_{[-3r^2/5,0]\times B_{3r/5}(x,g_t)}|\Ric| \le C(n){\sqrt{S}}r^{-1}.
	\end{align}
	Using the derivative version of evolution equation for $\Ric$ curvature in\eqref{e:evolution_Ricci} and using Moser's iteration, we get higher derivative estimates for $\Ric$
	\begin{align}
		\sup_{[-r^2/2,0]\times B_{r/2}(x,g_t)}|\nabla^k\Ric|\le C(n,k)\sqrt{S}r^{-1-k},~~~\mbox{ for $k\ge 0$.}
	\end{align}
	Noting that (see \cite{To,Ha})
	\begin{align}
		\partial_t\Rm=\Ric\ast \Rm+\nabla^2\Ric,
	\end{align}
	we have
	\begin{align}
		\sup_{[-1/2r^2,0]\times B_{r/2}(x,g_t)}|\partial_t\Rm|\le C(n)\sqrt{S}r^{-3}.	\end{align}
	On the other hand, we have
		\begin{align}
			\Big|\partial_t|\Rm|\Big|\le C(n)|\partial_t g|\cdot |\Rm|+C(n)|\partial_t\Rm|\le C(n)|\Ric|\cdot |\Rm|+C(n)|\partial_t\Rm|.
		\end{align}
		Thus
		\begin{align}
			\sup_{[-r^2/2,0]\times B_{r/2}(x,g_t)}\Big|\partial_t|\Rm|\Big|\le C(n)\sqrt{S}r^{-3},
		\end{align}
		which finishes the proof.
\end{proof}

\subsubsection{Proving Proposition \ref{p:backward}}

 Denote $f(t)= \sup_{(M,g(t))}|\Rm|$. Let us begin with the following claim. \\

\textbf{Claim:} There exists $\epsilon(n)>0$ such that if $r\le \epsilon(n)$ and if $f(s)\le r^{-2}$ with fixed $s\in [-\frac{1}{2}+2r^2,0]$, then
$$\sup_{-r^2+s\le t\le s}f(t)\le 2r^{-2}.$$
\\

Once we prove the claim, Proposition \ref{p:backward}  would follow by choosing $s=0$ in the claim. Now it will suffice to prove the claim.

	We prove this by induction on $r=r_k=2^{-k}$. Since $M$ is compact, there exists $k_0$ depending on $(M,g(t))$ with $t\in [-3/4,0]$ such that $\sup_{-3/4\le t\le 0}\sup_{(M,g(t))}|\Rm|\le r_{k_0}^{-2}$. Thus the claim is true for such $r=r_{k_0}$.

	Now we assume the claim holds for all $r\le r_i$ with $r_i\le \epsilon$ for some $\epsilon(n)$ to be determined. We are going to prove the claim for $r=r_{i-1}$ provided $r_{i-1}\le  \epsilon$.
	
	Assume $f(s)\le r^{-2}_{i-1}$ for some $-\frac{1}{2}+2r_{i-1}^2\le s\le 0$. By induction for $f(s)\le r^{-2} _i$ we have
	$$\sup_{-r_i^2+s\le t\le s}f(t)\le 2 r_i^{-2}.$$
By induction again for $s'=s-r_i^2$, we get
\begin{align}
	\sup_{-r_{i+1}^2-r_i^2+s\le t\le s-r_i^2}f(t)=\sup_{-5r_{i+1}^2+s\le t\le s-r_i^2}f(t)\le 2 r_{i+1}^{-2}.
\end{align}	
Using Lemma \ref{l:time_derivative_estimate}, we have derivative estimates
\begin{align}
	\sup_{-r_i^2+s\le t\le s}\Big|\partial_t|\Rm|\Big|\le C(n) r_{i+1}^{-3}.
\end{align}
Hence for all $-r_i^2+s\le t\le s$, we have
\begin{align}
	f(t)\le f(s)+(s-t)C(n)r_{i+1}^{-3}\le r_{i-1}^2+C(n)r_i^2r_{i+1}^{-3}\le \Big(1+16C(n)r_{i-1}\Big)r_{i-1}^{-2}.
\end{align}
By choosing $\epsilon(n)$ small such that $16\epsilon C(n)\le 1/8$, we arrive at
\begin{align}
	\sup_{-r_i^2+s\le t\le s}f(t)\le \frac{9}{8}r_{i-1}^{-2}\le r_i^{-2}.
\end{align}
By using the same argument as above to $s\leftarrow s-r_i^2$, we can show that
\begin{align}
	\sup_{-2r_i^2+s\le t\le s-r_i^2}f(t)\le \frac{10}{8}r_{i-1}^{-2}\le r_i^{-2}.
\end{align}
Using the same argument twice to $s\leftarrow s-2r_i^2$ and $s\leftarrow s-3r_i^2$, we have
\begin{align}
	\sup_{-4r_i^2+s\le t\le s}f(t)\le \frac{12}{8}r_{i-1}^{-2}\le 2r_{i-1}^{-2}.
\end{align}
We should point out that we can use induction just because $s-4r_i^{2}\ge 2r_i^2-\frac{1}{2}$ for any $s\ge -\frac{1}{2}+2r_{i-1}^2$.
Thus we finish the proof of the claim by choosing $\epsilon\le \frac{1}{128C(n)}$, where $C(n)$ is the constant in Lemma \ref{l:time_derivative_estimate}. In particular, by letting $s=0$ and $r=\epsilon A^{-1/2}$ with $\epsilon=\epsilon(n)$ in the claim, we conclude the backward curvature estimate. \qed

\vspace{.5cm}
\subsubsection{Proving Theorem \ref{t:curvature_derivative}}
Theorem \ref{t:curvature_derivative} follows now directly from Proposition \ref{p:backward} and Shi's curvature estimates \cite{Shi}.  \qed

\vspace{.5cm}
\subsection{$\epsilon$-regularity with bounded Ricci curvature}\label{ss:eps_regularity_assume_bounded_Ricci_Ricciflow}
In this subsection, we will prove an $\epsilon$-regularity estimate of Ricci flow provided bounded Ricci curvature at time zero. In subsection \ref{ss:bounds_Ricci_Ricciflow}, we will show the Ricci curvature is bounded. The main result of this subsection is the following estimate.
\begin{proposition}\label{p:eps_regularity_Ricci_flow_assume_bounded_Ricci}
	There exist universal constants $\epsilon$ and $C$ such that if $(M^4,g(t))$ is a compact Ricci flow on $(-1,1)$ with bounded scalar curvature $|R|\le 1$ at all time, and if at the time $t=0$ we have  $\sup_{(M,g(0))}|\Ric|\le 3$,
		$\int_{M}|\Rm|^2\le \epsilon$ and $\diam(M,g(0))\ge 1$,
	then
	\begin{align}
		\sup_{(M,g(0))}|\Rm|\le C.
	\end{align}
\end{proposition}
\begin{proof}
	For any $\epsilon'$ there exit $\epsilon=\epsilon(\epsilon')$ and $\eta=\eta(\epsilon')>0$ such that if $\int_M|\Rm|^2\le \epsilon$ then by Theorem \ref{t:curvatureL2_small} we have for every $x\in (M,g(0))$  and some $1>s\ge \eta>0$ that
	\begin{align}
		s^4\fint_{B_s(x)}|\Rm|^2\le \epsilon'.
	\end{align}
	This will be good enough to deduce the curvature estimates. Indeed, denote $|\Rm|(x_0)=\sup_{(M,g(0))}|\Rm|=A$. We will show that $A\le s^{-2}$. Otherwise,  by curvature derivative estimate in Theorem \ref{t:curvature_derivative} we have $\sup_{(M,g(0))}|\nabla \Rm|\le CA^{3/2}$. Therefore, by the volume comparison we have for a universal constant $C_0$ that
	\begin{align}
		C_0\le A^{-2}\fint_{B_{A^{-1/2}}(x_0)}|\Rm|^2\le C(n)s^4\fint_{B_s(x_0)}|\Rm|^2\le C\epsilon',
	\end{align}
	which leads to a contradiction by choosing a small and universal $\epsilon'$. Hence we finish the proof of Proposition \ref{p:eps_regularity_Ricci_flow_assume_bounded_Ricci}.
\end{proof}

\vspace{.5cm}
\subsection{Ricci curvature estimates }\label{ss:bounds_Ricci_Ricciflow}
In this subsection we will argue as Section \ref{s:soliton}  to show that the Ricci curvature is bounded if the $L^2$ curvature integral is small. The main result is the following
\begin{proposition}\label{p:Ricci_curvaturebounds_Ricciflow}
	There exist universal constants $\epsilon$ and $C$ such that if $(M^4,g(t))$ is a compact Ricci flow on $(-1,1)$ with bounded scalar curvature $|R|\le 1$ at all time, and if at the time $t=0$ we have $\diam(M,g(0))\ge 1$ and
	\begin{align}
		\int_{M}|\Rm|^2\le \epsilon,
	\end{align}
	then the Ricci curvature at $t=0$ is bounded
	\begin{align}
		\sup_M |\Ric|\le C.
	\end{align}
\end{proposition}
\begin{proof}

Let us fix $\epsilon$ as in Proposition \ref{p:eps_regularity_Ricci_flow_assume_bounded_Ricci} and denote the constant $C$ in Proposition \ref{p:eps_regularity_Ricci_flow_assume_bounded_Ricci} by $C_0$. Assume $\sup_{(M,g(0))}|\Ric|=A>3$. We will show $A$ is bounded by a universal constant $C$.

Consider the rescaling flow $\tilde{g}(t)=Ag(tA^{-1})$ which satisfies the condition in Proposition \ref{p:eps_regularity_Ricci_flow_assume_bounded_Ricci} with scalar curvature estimate $|\tilde{R}|\le A^{-1}$ and $\sup_{(M,\tilde{g}(0))}|\tilde{\Ric}|=1$. By Proposition \ref{p:eps_regularity_Ricci_flow_assume_bounded_Ricci} we have
\begin{align}
	\sup_{(M,\tilde{g}(0))}|\tilde{\Rm}|\le C_0
\end{align}
By Backward curvature estimate in Proposition \ref{p:backward} we have for a dimensional constant $\epsilon_0$ that
\begin{align}
	\sup_{-\epsilon_0\le t\le 0}\sup_{(M,\tilde{g}(t))}|\tilde{\Rm}|\le 2C_0.
\end{align}
Applying Ricci curvature estimate in Lemma \ref{l:time_derivative_estimate} we conclude for a universal constant $C_1$ that
\begin{align}
	1=\sup_{(M,\tilde{g}(0))}|\tilde{\Ric}|^2\le C_1C_0 A^{-1},
\end{align}
which in particular implies $A\le C_1C_0\equiv C$. Thus we finish the whole proof.
\end{proof}

\vspace{.5cm}
\subsection{Proving Theorem \ref{t:eps_regular_Ricciflow_global}}
Theorem \ref{t:eps_regular_Ricciflow_global} follows directly from  Proposition \ref{p:eps_regularity_Ricci_flow_assume_bounded_Ricci} and Proposition \ref{p:Ricci_curvaturebounds_Ricciflow}. \qed

\vspace{.5cm}
\section{Examples}\label{s:example}

In this section, we construct several examples to explain our results. The topology of all these examples are simple and are just product spaces.

\subsection{Higher dimensional Ricci flat metrics}
In this subsection, we will construct Ricci flat manifolds with dimension $n\ge 5$ such that they hold no $\epsilon$-regularity as Theorem \ref{t:cheeger-Tian}.

Indeed, for any $a>0$, we choose a sequence of complete Ricci flat $n-1$-dimensional manifolds $N_a$ with curvature $|\Rm|(p_a)=\max_{x\in N_a} |Rm|=a$. Consider the product manifolds $M_a=N_a\times S^1_\theta$ with product metric where $S^1_\theta$ is a circle with radius $\theta$. Then $M_a$ is Ricci flat.  Let us consider the ball $B_{1}(\bar p_a)$ where $\bar p_a= (p_a,0)\in M_a$. For any $\epsilon>0$, by choosing $\theta=\theta(a,\epsilon)$ small, due to $\Vol(B_1(\bar p_a))\le C(n)\theta \Vol(B_1(p_a))$ we always have
\begin{align}
	\int_{B_1(\bar p_a)}|\Rm|^{\frac{n}{2}}\le \epsilon,
\end{align}
which in particular implies that no $\epsilon$-regularity holds on $M_a$ since $|\Rm|(\bar p_a)\ge a$. $\square$

\begin{remark}
	For $n=5$, we can choose $4$-dimensional manifolds $N_a^4$ as the blow down of Eguchi-Hanson manifolds.
\end{remark}
\begin{remark}
From the construction above, we see that even $\int_{B_1}|\Rm|^p\le \epsilon$ for $p>n/2$ there still exists no $\epsilon$-regularity.	 Indeed, as pointed out by Naber-Zhang \cite{NaZh}, there are topological restrictions on higher dimensional $\epsilon$-regularity.
\end{remark}

\subsection{K\"ahler metric with zero scalar curvature }\label{ss:kahler_zero_scalar}
In this subsection, we construct K\"ahler metrics with zero scalar curvature which doesn't satisfy an $\epsilon$-regularity.

Indeed, let $(N,h)$ be a complete noncompact hyperbolic surface with finite volume and constant curvature $-1$ (see \cite{FaMa} for the existence of such a hyperbolic surface). For any $a>0$ consider the product space $(M_a,g_a)=N_a\times S_a^2$ with product metric, where $(N_a,h_a)=(N,a^2h)$ and $S_a^2$ is the standard two sphere with curvature $a^{-2}$. Since $N_a$ and $S_a^2$ are K\"ahler we have $M_a$ is K\"ahler and scalar flat. Let us see that $(M_a,g_a)$ doesn't satisfies any $\epsilon$-regularity as Cheeger-Tian. In fact, for any $a>0$ and $\epsilon>0$ since $M_a$ is noncompact and has finite volume, there exists $B_1(x)\subset M_a$ such that $\Vol(B_1(x))\le a^4 \epsilon$. Therefore
\begin{align}
	\int_{B_1(x)}|\Rm|^2\le \Vol(B_1(x)) a^{-4}\le \epsilon.
\end{align}
On the other hand, noting that the curvature of $(M_a,g_a)$ satisfies $|\Rm|(x)=a^{-2}$, there exists no $\epsilon$-regularity if $a$ is sufficiently small.$\square$
\\
\begin{remark}
K\"ahler metric with zero scalar curvature is a critical metric as defined in Definition \ref{d:critical}.	
\end{remark}



\subsection{Critical metrics with bounded Ricci curvature}
If we assume a bounded Ricci curvature for critical metrics, we can prove
\begin{theorem}\label{t:critical_regularity}
	Let $(M^4,g)$ equip with a critical metric $g$. Then there exist $\epsilon>0$ and $C$ such that if for some $r\le 1$ that
	\begin{align}
		\int_{B_r(p)}|Rm|^2\le \epsilon, \mbox{  and  }\sup_{B_r(p)}|\Ric|\le 3
	\end{align}
	then we have
	\begin{align}\label{e:curvature_estimate}
		\sup_{B_{\frac{1}{2}r}(p)}|Rm|\le Cr^{-2},
	\end{align}
\end{theorem}
\begin{proof}
	This is a direct consequence of the curvature estimate in Theorem \ref{t:curvatureL2_small} and the $\epsilon$-regularity theorem \ref{t:eps_regularity_critical_metric} for critical metrics.
\end{proof}
\begin{remark}
By the example in Subsection \ref{ss:kahler_zero_scalar}, the bounded Ricci curvature condition is a necessary condition in Theorem \ref{t:critical_regularity}.
\end{remark}
  \begin{remark}
 The constant $C$ in \eqref{e:curvature_estimate} can't be chosen arbitrarily small. Indeed, consider $M=S^3\times S_\theta^1$ with product metric where $S^3$ is the standard sphere and $S_\theta^1$ is a circle with radius $\theta$. We have $\Delta \Ric=0$ which means $M$ has critical metric. 	By letting $\theta$ small we have the integral $\int_{B_r}|\Rm|^2$ is small as we want. However the curvature $|\Rm|$ has a definite lower bound. That means the constant $C$ in \eqref{e:curvature_estimate} has a definite lower bound for any $\epsilon>0$.
  \end{remark}



\section{Further discussions on Ricci flow}\label{s:discussion}
The key ingredient for $\epsilon$-regularity of Ricci flow is the curvature derivative in Theorem \ref{t:curvature_derivative} which relies on a backward curvature estimate. Our backward curvature proof of Proposition \ref{p:backward} depends on the closeness of the manifold. If one can give a local proof of backward curvature estimate, then an $\epsilon$-regularity follows easily by the same argument as in Section \ref{s:eps_Ricciflow}. The backward estimate can be stated in the following way:

\textbf{Backward Pseudolocality Estimate:} Let $(M^4,g(t))$ be a $4$-dimensional complete Ricci flow $\partial_t g=-2\Ric$ on $(-1,1)$ with bounded scalar curvature $|R|\le 1$. Then there exists $\epsilon$ such that if $\sup_{B_r(x,g_0)}|\Rm|\le r^{-2}$ for some $r\le 1$ at time $t=0$  we have
\begin{align}\label{e:backward_assumption}
	\sup_{(y,t)\in B_{r/2}(x,g_0)\times [-\epsilon r^2,0]}|\Rm|(y,t)\le 2r^{-2}.
\end{align}
\\

The backward Pseudolocality estimate is roughly a space time regularity scale estimate; see also \cite{HeNa} for regularity scale estimate in their setting. Under volume noncollapsed assumption along the flow, the backward Pseudolocality estimate \eqref{e:backward_assumption} was proved in \cite{BaZh}.



\section*{Acknowledgements}
Both authors would like to thank Professor Aaron Naber and Gang Tian for insightful discussions and useful comments for the paper. The second author would also like to thank Professor Peter Topping and Kai Zheng for several helpful discussions during this work. The first author was supported by National Natural Science Foundation of China under grant (No.11501027). The second author was supported by the EPSRC on a Programme Grant entitled "Singularities of Geometric Partial Differential Equations" reference number EP/K00865X/1.

\bibliographystyle{plain}

\end{document}